\documentclass[11pt,reqno]{amsart}

\usepackage{amsfonts}
\usepackage{amssymb}
\usepackage{amsthm}
\usepackage{amsmath}

\theoremstyle{plain}
 \newtheorem{theorem}{\bf Theorem}[section]
 \newtheorem{lemma}[theorem]{Lemma}
 \newtheorem{corollary}[theorem]{Corollary}
 \newtheorem{claim}[theorem]{Claim}
\theoremstyle{definition}
 \newtheorem{definition}[theorem]{Definition}

\theoremstyle{remark}

\numberwithin{equation}{section}

\newcommand{\TS}{\textstyle}

\renewcommand{\a}{\alpha}
\renewcommand{\b}{\beta}
\newcommand{\g}{\gamma}
\renewcommand{\d}{\delta}
\renewcommand{\k}{\kappa}

\renewcommand{\o}{\omega}

\newcommand{\ps}{\emptyset}
\newcommand{\sse}{\subseteq}
\newcommand{\f}{\Vdash}
\newcommand{\sledi}{\Rightarrow}
\newcommand{\rest}{\upharpoonright}
\newcommand{\akko}{\Leftrightarrow}

\newcommand{\cp}{\mathcal P}
\newcommand{\cg}{\mathcal G}

\newcommand{\cd}{\mathcal D}

\newcommand{\seq}[1]{\left<#1\right>}
\newcommand{\set}[1]{\left\{#1\right\}}
\newcommand{\abs}[1]{\left\vert#1\right\vert}

\newcommand{\ran}{\operatorname{ran}}

\newcommand{\trcl}{\operatorname{trcl}}
\newcommand{\rank}{\operatorname{rank}}

\newcommand{\cf}{\operatorname{cof}}

\newcommand{\tr}{\operatorname{Tr}}
\newcommand{\supp}{\operatorname{supp}}
\newcommand{\cl}{\operatorname{cl}}
\renewcommand{\int}{\operatorname{int}}
\newcommand{\ord}{\operatorname{Ord}}
\newcommand{\izo}{\overset{\cong}{\longrightarrow}}

\author[B. Kuzeljevic]{Borisa Kuzeljevic}
\address{Mathematical Institute SANU, Kneza Mihaila 36, 11001 Belgrade, Serbia}
\curraddr{}
\email{borisa@mi.sanu.ac.rs}
\urladdr{}
\dedicatory{}

\author[S. Todorcevic]{Stevo Todorcevic}
\address{Department of Mathematics, University of Toronto, Toronto, Canada, M5S 2E4. Institut de Math\'ematiques de Jussieu, UMR 7586, 2 pl. Jussieu, Case 7012, 75251 Paris Cedex 05, France. Mathematical Institute SANU, Kneza Mihaila 36, 11001 Belgrade, Serbia}
\curraddr{}
\email{stevo@math.toronto.edu, stevo.todorcevic@imj-prg.fr,
stevo@mi.sanu.ac.rs}
\urladdr{}
\dedicatory{}

\title[Forcing with matrices of countable elementary submodels]{Forcing with matrices of countable elementary submodels}
\thanks{ This paper was conceived during the Summer of 2014 when the first author was visiting Institut de Math\'ematiques de Jussieu. The first author would like to thank
the Equipe de Logique of that Institute for support which made this visit possible}
\date{\today}

\begin{document}

\begin{abstract}
We analyze the forcing notion $\mathcal P$ of finite matrices whose rows consists of isomorphic countable elementary submodels of a given structure of the form $H_{\theta}$. We show that forcing with this poset adds a Kurepa tree $T$. Moreover, if $\mathcal P_c$ is a suborder of $\mathcal P$ containing only continuous matrices, then the Kurepa tree $T$ is almost Souslin, i.e. the level set of any antichain in $T$ is not stationary in $\omega_1$.
\end{abstract}

\maketitle

\section{Introduction}

In this paper we analyze the forcing notion mentioned in the remark on page 217 of \cite{pfa}.  This is the forcing notion of finite  matrices
whose rows consists of isomorphic countable elementary submodels of a given structure  of the form $H_\theta.$ In \cite{pfa} they were merely
meant as side conditions to various proper forcing constructions when one is interested in getting the $\aleph_2$-chain condition that can be iterated.
Soon afterwards the second author realized that this variation of the original side conditions is as much an interesting forcing notion as the poset of
finite chains of countable elementary submodels analyzed briefly in Theorem 6 of \cite{pfa}. For example,  he showed that the poset of finite matrices
of row-isomorphic countable elementary sub-models always forces CH (so, in particular preserves CH if it is true in the ground model). The second author observed
at the same time that this forcing notion gives a natural example of a Kurepa tree. Here we shall explore this further and produce a natural variation of this forcing notion
that gives us a Kurepa tree with no stationary antichains.  This gives us a quite different forcing construction of such a tree from the previous ones which use
countable rather than finite conditions (see  \cite{golshani} and \cite{note}). We believe that there will be other natural variations of this forcing notion with interesting applications. For example, we note that  in the recent paper \cite{aspero}  Aspero and Mota have used the poset of  finite matrices of elementary submodels  to control their iteration scheme  which shows that the forcing axiom for the class of all finitely proper posets of size $\o_1$ is compatible with $2^{\aleph_0}>\aleph_2$.
In view of the recent efforts to generalize the side condition method to higher cardinals (see \cite{neeman,neemanslides}) it would be interesting to also explore the possible higher-cardinal analogues of the
posets that we analyze here. This could also be asked for the original side-condition poset of finite elementary chains of countable elementary submodels of \cite{pfa} which, as shown in Theorem 6
of \cite{pfa},  gives us a natural forcing notion that collapses a given cardinal $\theta$ to $\omega_1$ preserving all other cardinals\footnote{It should be noted that the hypothesis of Theorem 6 of \cite{pfa} assumes that
there is a stationary subset $S$ of $[\theta]^{\aleph_0}$ or cardinality $\theta$,  a condition that is satisfied by many cardinals and in particular by all cardinals of uncountable cofinality if $0^\sharp$ does not exist }. As far as we know, no higher-cardinal analogue of this poset has been
produced. Some research related to this have been recently produced by Aspero \cite{aspero1}.

\subsection*{Acknowledgements:} The authors would like to thank the referee for careful reading the paper and many useful suggestions.

\section{Preliminaries}\label{matrix}

\begin{definition}\label{d1}
Let $\theta\ge \o_2$ be a regular cardinal. By $H_\theta$ we denote the collection of all sets whose transitive closure has cardinality $<\theta.$ We consider it as a model of the form $(H_\theta, \in, <_\theta)$ where
$<_\theta$ is some fixed well-ordering of $H_\theta$ that will not be explicitly mentioned.  The partial order $\cp$ is the set of all functions $p:\o_1\to H_{\theta}$ satisfying:
\begin{enumerate}
\item $\supp(p)=\set{\a<\o_1:p(\a)\neq \ps}$ is a finite set;
\item $p(\a)$ is a finite collection of isomorphic countable elementary submodels of $H_{\theta}$ for every $\a\in\supp(p)$;
\item for each $\a,\b\in\supp(p)$ if $\a<\b$ then $\forall M\in p(\a)\ \exists N\in p(\b)\ M\in N$;
\end{enumerate}

\noindent The ordering on $\cp$ is given by:
\begin{equation}\label{eq1}
p\le q\ \akko\ \forall \a<\o_1\ \ q(\a)\sse p(\a).
\end{equation}
\end{definition}

\noindent
The fact that $M$ is a countable elementary submodel of $\seq{H_{\theta},\in}$ will be denoted by $M\prec H_{\theta}$. Also, if $M\prec H_{\theta}$, then $\overline M\in H_{\o_1}$ denotes the transitive collapse of $M$ with $\pi_M$ being the corresponding isomorphism. For $p\in \cp$ and $\a\in \supp(p)$ we denote $\delta_{\alpha}^p=M\cap \o_1$ where $M$ is some (any) model in $p(\a)$. Also, if $M\prec H_{\theta}$, then $\d_M$ will denote the ordinal $M\cap \o_1$. We list some standard  lemmas concerning countable elementary submodels of $H_{\theta}$ that will be useful throughout  the paper.

\begin{lemma}\label{t20}
Let $F$ be a countable subset of $H_{\theta}$. Then the set of all ordinals of the form $M\cap \o_1$ such that $M$ is a countable elementary submodel of $H_{\theta}$ with $F\sse M$ contains a club.
\end{lemma}

\begin{lemma}\label{t23}
If $M\prec H_{\theta}$ contains some element $X$, then $X$ is countable if and only if $X\sse M$.
\end{lemma}

\begin{lemma}\label{t24}
If $M\prec H_{\theta}$ contains as an element some subset $A$ of $\o_1$, then $A$ is uncountable if and only if $A\cap M\cap \o_1$ is unbounded in $M\cap \o_1$.
\end{lemma}

\begin{lemma}\label{t25}
If $M\prec H_{\theta}$, $X$ is in $H_{\theta}$ and $X$ is definable from parameters in $M$, then $X\in M$.
\end{lemma}

\begin{lemma}\label{t26}
Let $\seq{N_{\xi}:\xi<\o_1}$ be a continuous $\in$-chain of countable elementary submodels of $H_{\theta}$. Then $\set{\xi<\o_1:N_{\xi}\cap \o_1=\xi}$ is a club in $\o_1$.
\end{lemma}

\begin{lemma}\label{t9}
Let $M_0$ and $M_1$ be isomorphic countable elementary submodels of some $H_{\theta}$. Let $L_0=M_0\cap \o_2$ and $L_1=M_1\cap \o_2$. Then $L_0\cap L_1$ is an initial segment of both $L_0$ and $L_1$.
\end{lemma}

\begin{proof}
Without loss of generality, we can assume that both $M_0$ and $M_1$ contain the same family of mappings $\seq{e_{\g}:\g<\o_2}$ where each map $e_{\g}:\g\to \o_1$ is 1-1. Now we will prove that if $\b$ is in $L_0\cap L_1$ and $\a<\b$ is in $L_0$ then $\a\in L_1$. Consider the map $e_{\b}:\b\to\o_1$. Then, in $M_0$ there is some $\xi<\o_1$ such that $e_{\b}(\a)=\xi$, i.e. $\a=e_{\b}^{-1}(\xi)$. But $M_1$ knows both $e_{\b}$ and $\xi$ (this is because $M_0\cong M_1\sledi M_0\cap \o_1=M_1\cap \o_1$). Hence, $\a\in M_1\cap \o_2=L_1$.
\end{proof}

If $\cg\sse \cp$ is a filter in $\mathcal P$ generic over $V$, then we define $G:\o_1\to H_{\theta}$ as the function satisfying
\[\TS
G(\a)=\set{M\prec H_{\theta}: \exists p\in \cg\ \mbox{such that}\ M\in p(\a)}.
\]
Note that $G$ is a well defined function because $\cg$ is a filter. For $\a$ in the domain of $G$ we denote $\d_{\a}=M\cap \o_1$ for some (any) $M$ in $G(\a)$. Further, we denote $A_{\gamma}=\bigcup_{M\in G(\gamma)}M\cap \o_2$ for $\g<\o_1$, and note that if $\g<\delta$, then $A_{\gamma}\sse A_{\delta}$. Also, we define the function $g:\o_1\to H_{\o_1}$ with $g(\a)=\overline M$ for some (any) model $M$ from $G(\a)$ and for $p\in \cp$, by $\bar p$ we define $\bar p:\o_1\to [H_{\o_1}]^{\o}$ as a function with the same support as $p$ which maps $\a\in\supp(\bar p)$ to the transitive collapse of some model from $p(\a)$, while for $\a\in \o_1\setminus \supp(\bar p)$ take $\bar p(\a)=\ps$.

\begin{lemma}\label{t14}
Let $p,q\in \cp$. If $\bar p=\bar q$, then $p$ and $q$ are compatible conditions.
\end{lemma}

\begin{proof}
First note that if $\bar p=\bar q$, then $\supp(p)=\supp(q)$. Also, if two countable elementary submodels of $H_{\theta}$, say $M_1$ and $M_2$ have the same transitive collapse, then they are isomorphic (the isomorphism is simply $\pi_{M_2}^{-1}\circ\pi_{M_1}$). Now it is clear that the function $r:\o_1\to H_{\theta}$ defined by $r(\a)=p(\a)\cup q(\a)$ is in $\cp$ and that it satisfies $r\le p,q$.
\end{proof}

For $p,q\in \cp$ we will define their 'join' $p\lor q$ as the function from $\o_1$ to $H_{\theta}$ satisfying $(p\lor q)(\a)=p(\a)\cup q(\a)$ for $\a<\o_1$. Further, if a condition $q\in \cp$ and $M\prec H_{\k}$ ($\d_M=M\cap \o_1$) for $\k\ge\theta$ are given, it is clear what intersection $q\cap M$ represents, and we define the restriction of $q$ to $M$ as a function with finite support $q\mid M:\o_1\to H_{\theta}$ satisfying $\supp(q\mid M)=\supp(q)\cap\delta_M$ and for $\a<\d_M$:
\[\textstyle
(q\mid M)(\a)=\set{\varphi_{M'}(N): M'\in q(\d_M),\ \varphi_{M'}:M'\izo M\cap H_{\theta},\ N\in q(\a)\cap M'}.
\]
Note that the function $q\mid M$ is in $\cp$. We will also need the following notion which we call 'the closure of $p$ below $\d$'.

\begin{definition}\label{d2}
Let $p\in \cp$ and $\d\in \supp(p)$. Then $\cl_{\d}(p):\o_1\to H_{\theta}$ is a function such that $\supp(\cl_{\d}(p))=\supp(p)$ and $\cl_{\d}(p)(\g)=p(\g)$ for $\g\ge \d$, while for $\g<\d$ we have
\[\textstyle
\cl_{\d}(p)(\g)\!=\!\set{\psi_{N_1,N_2}(M):M\!\in\! p(\g)\cap N_1\ \!\&\!\ N_1,N_2\!\in\! p(\d)\ \!\&\!\ \psi_{N_1,N_2}:N_1\izo N_2}.
\]
\end{definition}

We will also need the following standard  lemmas later in the paper.

\begin{lemma}\label{t17}
Suppose that $\theta\ge \o_2$ is a regular cardinal. If $\cp\in H_{\theta}$ then in $V[\cg]$ we have $H_{\theta}[\cg]=\set{\int_{\cg}(\tau):\tau\in H_{\theta}\ \mbox{is a }\cp\mbox{-name}}=H_{\theta}^{V[\cg]}$.
\end{lemma}

\begin{proof}
$H_{\theta}^{V[\cg]}\sse H_{\theta}[\cg]$ follows from the fact that any $x\in H_{\theta}^{V[\cg]}$ is of the form $\int_{\cg}(\tau)$ for some $\tau\in H_{\theta}$ (see \cite[Claim I 5.17]{proper}). If $\tau\in H_{\theta}$ is a $\cp$-name, then from $\abs{\trcl(\tau)}<\theta$ and $\rank(\int_{\cg}(\tau))\le\rank(\tau)$ we have that $\abs{\trcl(\int_{\cg}(\tau))}<\theta$.
\end{proof}

\begin{lemma}\label{t18}
If $M\prec H_{\theta}$ and $\cp\in M$, then in $V[\cg]$ we have that
\[\textstyle
M[\cg]=\set{\int_{\cg}(\tau):\tau\in M\ \mbox{is a }\cp\mbox{-name}}\prec H_{\theta}[\cg]=H_{\theta}.
\]
\end{lemma}

\begin{proof}
First note that $M[\cg]\sse H_{\theta}[\cg]$. Now take any $\tau_1,\dots,\tau_n\in M$ and any formula $\varphi(x,x_1,\dots,x_n)$. Assume that there is some $y\in H_{\theta}[\cg]$ such that $H_{\theta}[\cg]\models \varphi(y,\int_{\cg}(\tau_1),\dots,\int_{\cg}(\tau_n))$. Then there is $p\in \cg$ which forces $\varphi(\tau_y,\tau_1,\dots,\tau_n)$ for a $\cp$-name $\tau_y\!\in\! H_{\theta}$ and $M\!\prec\! H_{\theta}$ implies $p\f\! \varphi(\tau,\tau_1,\dots,\tau_n)$ for a $\cp$-name $\tau\in M$. Hence $H_{\theta}[\cg]\models\varphi(\int_{\cg}(\tau),\int_{\cg}(\tau_1),\dots,\int_{\cg}(\tau_n))$
for a $\cp$-name $\tau\in M$, so $M[\cg]\prec H_{\theta}[\cg]$.
\end{proof}

\section{Properness}

In this section we show that $\cp$ satisfies the condition stronger than being proper, namely $\cp$ is strongly proper forcing. We are not sure who was the first do define this notion, but Mitchell's paper  \cite{mitchell} is a good reference.

\begin{definition}\label{d3}
If $P$ is a forcing notion and $X$ is a set, then we say that $p$ is strongly $(X,P)$-generic if for any set $D$ which is dense and open in the poset $P\cap X$, the set $D$ is predense in $P$ below $p$. The poset $P$ is strongly proper if for every large enough regular cardinal $\k$, there are club many countable elementary submodels $M$ of $H_{\k}$ such that whenever $p\in M\cap P$, there exists a strongly $(M,P)$-generic condition below $p$.
\end{definition}

\begin{lemma}\label{T1}
$\cp$ is strongly proper.
\end{lemma}

\begin{proof}
Let $p$ be a condition in $\cp$ and $M\prec H_{\k}$ (for some $\k\ge \theta$ which is large enough) such that $p,\cp\in M$. Denote $\delta=M\cap \o_1$. We will show that the condition $p_M\!=\!p\cup \set{\seq{\d,M\cap H_{\theta}}}$ is strongly generic (i.e. if $q\le p_M$ and $\cd\sse\cp\cap M$ is dense open, then there is some $q'\in \cd$ such that $q\not\perp q'$). Let $q\le p_M$ be arbitrary and $\mathcal D\sse M\cap \cp$ dense open. Consider the condition $q\mid M\in \cp \cap M$. Because $\cd\sse M$ is dense open, there is some $q'\le q\mid M$ which is in $\cd$ (hence in $M$). To finish the proof, we still have to show that $q$ and $q'$ are compatible. First note that $\supp(q\mid M)=\supp(q)\cap \supp(q')$ and consider the following function $r:\o_1\to H_{\theta}$ defined on the support $\supp(r)=\supp(q)\cup\supp(q')$:

\noindent
For $\a\ge \d$ define $r(\a)=q(\a)$ and for $\a<\d$ define
\[\textstyle
r(\a)\!=\!(q\lor q')(\a)\cup\set{\varphi^{-1}_{N'}(N):N\!\in\! q'(\a), N'\in q(\d),\ \varphi_{N'}:N'\!\izo\! M\!\cap\! H_{\theta}}.
\]
We will prove that $r\le q,q'$ is in $\cp$ which will finish the proof. Properties (1) and (2) from Definition \ref{d1} are clear. To see (3) take any $\a,\b<\o_1$ with $\a<\b$ and any $N\in r(\a)$.

If $\a\ge \d$, then $r(\a)=q(\a)$ and $r(\b)=q(\b)$, hence there is clearly some $N'\in r(\b)=q(\b)$ such that $N\in N'$.

If $\a<\d$ and $\b\ge \d$, then $N$ belongs to some $N'\in q(\d)$ which belongs to some $N_1\in q(\b)=r(\b)$, hence $N\in N_1$ and the statement is true in this case as well.

Otherwise, note that $\supp(q)\cap\d\sse \supp(q')$ and consider the following two possibilities. The first, that $N\in q'(\a)$ and $\b<\d$. Then there is a model $N'\in q'(\b)\sse r(\b)$ such that $N\in N'$. The second case is that $N\in r(\a)\setminus q'(\a)$ and $\b<\d$. If $N=\varphi^{-1}_{N'}(N_1)$ for some $N_1\in q'(\a)$, take $N_2\in q'(\b)$ such that $N_1\in N_2$ and note that $N=\varphi^{-1}_{N'}(N_1)\in \varphi^{-1}_{N'}(N_2)\in r(\b)$. If not, then $N\in q(\a)$, so there is some $N_3\in q'(\b)$ such that $\varphi_{N'}(N)\in N_3$, hence $N\in \varphi^{-1}_{N'}(N_3)\in r(\b)$ and the proof is finished.
\end{proof}

\begin{lemma}\label{T2}
Let $\cg$ be a filter generic in $\cp$ over $V$, let $M,M'\prec H_{(2^{\theta})^+}$ and $p,\cp\in M\cap M'$. If $\varphi:M\izo M'$ then for $\delta=M\cap \o_1=M'\cap \o_1$ the condition $p_{MM'}=p\cup\set{\seq{\delta,\set{M\cap H_{\theta},M'\cap H_{\theta}}}}$ satisfies:
\[\textstyle
p_{MM'}\Vdash \check\varphi[\dot\cg\cap \check M]=\dot\cg\cap \check M'.
\]
\end{lemma}

\begin{proof}
Assume the contrary, that there is a condition $q\le p_{MM'}$ and a set $x$ such that $q\f \check x\in \dot \cg\cap \check M\wedge\check \varphi(\check x)\notin \dot \cg\cap \check M'$. Then we have $q\f \check x\in \dot\cg\cap \check M$ and $q\f \check\varphi(\check x)\notin \dot \cg\cap \check M'$. From the fact that $q\f \check x\in \dot\cg\cap \check M$, we have that $q\not\perp x$ (this is true because if $q\perp x$, then it is not possible that $q\f \check x\in \dot \cg$). Now take some $q'\le q,x$ and assume that for all $r\le q'$ there is some $t\le r$ such that $t\le \varphi(x)$ (which implies that $t\f \check\varphi(\check x)\in \dot \cg$). It follows that the set $\set{t\in \cp: t\f \check\varphi(\check x)\in \dot \cg}$ is dense below $q'$, which is impossible because it would imply that $q'\f \check\varphi(\check x)\in \dot \cg\cap \check M'$ which is in contradiction with the assumption that $q'\le q$ and that $q\f \check\varphi(\check x)\notin \dot \cg\cap \check M'$. Hence, we can pick some $r\le q'\le q,x$ which is incompatible with $\varphi(x)$. Now, consider the condition $r\mid M$. We have the following claim.

\begin{claim}\label{t6}
Conditions $\varphi(r\mid M)$ and $r$ are compatible.
\end{claim}

\begin{proof}
First note that $\supp(\varphi(r\mid M))\sse \supp(r)$. We will prove that the condition $s=\varphi(r\mid M)\lor r$ is in $\cp$ which will prove the claim (then $s$ will be below both $\varphi(r\mid M)$ and $r$). It is clear that the conditions (1) and (2) from Definition \ref{d1} are fulfilled so pick arbitrary $\a,\b\in \supp(s)$ with $\a<\b$. If $\a\ge \d$ then every $N\in s(\a)$ is in $r(\a)$ hence there is some $N'\in r(\b)=s(\b)$ such that $N\in N'$. If $\b\le \d$ then for $N\in s(\a)\cap r(\a)$ there is some $N'\in r(\b)\sse r(s)$ such that $N\in N'$. On the other hand, if $\b\le \d$ and $N\in s(\a)\cap \varphi(r\mid M)(\a)$, then $N=\varphi(N_1)$ for some $N_1\in (r\mid M)(\a)$, and $N_1=\psi_{N_2,M}(N_3)$ for $N_2\in r(\d)$, $N_3\in r(\a)\cap N_2$ and an isomorphism $\psi_{N_2,M}:N_2\to M\cap H_{\theta}$. Now, there is some $N_4\in r(\b)$ such that $N_3\in N_4$. So, clearly we have that $N\in \varphi(\psi_{N_2,M}(N_4))\in\varphi(r\mid M)(\b)\sse s(\b)$. If $\a<\d$ and $\d\le \b$, then for each $N\in s(\a)\cap r(\a)$ there is some $N'\in r(\b)=s(\b)$ such that $N\in N'$. Finally, let $\a<\d$, $\d\le \b$ and $N\in s(\a)\cap \varphi(r\mid M)(\a)$. Then $N\in M\cap H_{\theta}\in s(\d)\cap r(\d)$ and clearly there is some $N'\in r(\b)=s(\b)$ such that $N\in M\cap H_{\theta}\in N'$, hence $N\in N'$.
\end{proof}

Now because $x\in M$ and $r\le x$ we have that $r\mid M\le x$ which implies $\varphi(r\mid M)\le \varphi(x)$ (the last implication is true because $\varphi$ is an automorphism between $M$ and $M'$). Now from $\varphi(r\mid M)\le \varphi(x)$ and $r\perp \varphi(x)$ follows that $\varphi(r\mid M)$ and $r$ are incompatible which is not possible according to the Claim \ref{t6}.
\end{proof}

The following lemma will be useful in Section \ref{kurepa} of the paper.

\begin{lemma}\label{t13}
If $M=M'\cap H_{\theta}$ for some $M'\prec H_{(2^{\theta})^+}$ such that $\cp\in M'$ and if $M\in G(\d)$ for $\d=M\cap \o_1$, then in $V[\cg]$ we have $M[\cg]\cap \o_1=\d$.
\end{lemma}

\begin{proof}
First note that any $p\in \cg$ such that $M\in p(\d)$ is an $(M',\cp)$-generic condition which forces that $M'\cap \ord=M'[\cg]\cap \ord$ (see \cite[Lemma III 2.6]{proper}). Because $\o_1\sse H_{\theta}$ this implies that $M[\cg]\cap \o_1=M\cap \o_1=\d$.
\end{proof}

\section{Preserving CH}

\begin{lemma}[CH]\label{t15}
$\cp$ satisfies $\o_2$-c.c.
\end{lemma}

\begin{proof}
Assume that CH holds and that there is a sequence $\set{p_{\a}:\a<\o_2}$ of pairwise incompatible elements of $\cp$. For each $\a<\o_2$ the function $\bar p_{\a}$ (the transitive collapse of $p_{\a}$) is a finite subset of $\o_1\times H_{\o_1}$. Hence, there are some distinct $\a,\b<\o_2$ such that $\bar p_{\a}=\bar p_{\b}$ (here we are using CH which implies that $\abs{H_{\o_1}}=\o_1$). But then conditions $p_{\a}$ and $p_{\b}$ are compatible by Lemma \ref{t14} which is a contradiction with the choice of the sequence $\set{p_{\a}:\a<\o_2}$.
\end{proof}

\begin{lemma}
$\cp$ preserves CH.
\end{lemma}

\begin{proof}
Assume that CH holds in $V$ and that there is a sequence $\set{\tau_{\a}:\a<\o_2}$ of $\cp$-names and a condition $p\in \cg$ such that
\[\textstyle
p\Vdash``\seq{\tau_{\a}:\a<\o_2}\ \mbox{is a sequence of pairwise distinct reals"}.
\]
For each $\a<\o_2$ let $M_{\a}$ be a countable elementary submodel of $H_{(2^{\theta})^+}$ containing $\cp,\tau_{\a},p$.

Now, using CH, we conclude that there are $\a,\b<\o_2$ ($\a\neq\b$) such that there is an automorphism $\varphi:M_{\a}\to M_{\b}$ which satisfies $\varphi(\tau_{\a})=\tau_{\beta}$. To see this consider the transitive collapses of $M_{\a}$'s, they are countable submodels of $H_{\o_1}$ but since we assumed the CH (which implies $\abs{H_{\o_1}}=\o_1$) there must be two collapses $\overline{M_{\a}}$ and $\overline{M_{\b}}$ which are isomorphic (via isomorphism $\phi$). Then the isomorphism $\varphi$ is given by $\varphi=\pi^{-1}_{M_{\b}}\circ \phi\circ \pi_{M_{\a}}$. Also, it clearly holds that $\varphi(M_{\a}\cap H_{\theta})=M_{\b}\cap H_{\theta}$, and because $M_{\a}$ and $M_{\b}$ are isomorphic we can denote $\delta=M_{\a}\cap \o_1=M_{\b}\cap \o_1$.

Now we prove that there is a condition $p_{\a\b}\le p$ such that $p_{\a\b}\f \tau_{\a}=\tau_{\b}$, hence $\tau_{\a}$ and $\tau_{\b}$ cannot be names for distinct reals in $V[\cg]$. First note that $\varphi$ being an isomorphism, we have
\begin{equation}\label{eq2}
\forall n<\o\ \forall p'\in M_{\a}\cap \cp\ \forall \epsilon<2\ (p'\f \tau_{\a}(\check n)=\check \epsilon\akko \varphi(p')\f \tau_{\b}(\check n)=\check \epsilon).
\end{equation}

\begin{claim}\label{t3}
If $p_{\a\b}=p\cup\set{\seq{\d,\set{M_{\a}\cap H_{\theta},M_{\b}\cap H_{\theta}}}}$, then $p_{\a\b}\f \tau_{\a}=\tau_{\b}$.
\end{claim}

\begin{proof}
Assume the contrary, that there is some $q\le p_{\a\b}$ and $n<\o$ such that $q\f \tau_{\a}(\check n)\neq \tau_{\b}(\check n)$ (suppose that $q\f\tau_{\a}(\check n)=\check 0$ and $q\f \tau_{\b}(\check n)=\check 1$). Then by elementarity of $M_{\a}$ there is some $r\in \cp \cap M_{\a}$ such that $r\le q\mid M$ and $\supp(q\mid M_{\a})\sqsubseteq \supp(r)$ (where $a\sqsubseteq b$ denotes that $a$ is an initial segment of $b$) and which satisfies $r\f \tau_{\a}(\check n)=\check 0$.  Again, because $\varphi$ is an isomorphism we have $\varphi(q\mid M_{\a})=q\mid M_{\b}$, so $q\mid M_{\b}$ is compatible with $\varphi(r)$. From $q\mid M_{\b}\not\perp \varphi(r)$ we conclude that $q\not\perp \varphi(r)$. But, then $\varphi(r)\f \tau_{\b}(\check n)=\check 0$ (from equation \ref{eq2}) which is in contradiction with the fact that $q\f \tau_{\b}(\check n)=\check 1$ (simply because $\varphi(r)\not\perp q$).
\end{proof}

Hence, according to the Claim \ref{t3}, $p$ cannot force that $\seq{\tau_{\a}:\a<\o_2}$ is a sequence of names for distinct reals in $V[\cg]$, which proves the theorem.\end{proof}

\section{Kurepa tree}\label{kurepa}

Recall that Kurepa tree is a tree of height $\o_1$, with all levels countable but at least $\o_2$ branches. In this section we will show that forcing with $\cp$ adds a Kurepa tree.

\begin{theorem}\label{t4}
There is a Kurepa tree in $V[\cg]$.
\end{theorem}

\begin{proof}
For each $\a<\o_2$ define the function $f_{\a}:\o_1\to \o_1$:
\begin{equation}\label{eq3}
f_{\a}(\d)=\left\{\begin{array}{rl}\xi,&\mbox{if there is }M\in G(\d)\ \mbox{such that}\ \a\in M,\ \pi_{M}(\a)=\xi;\\ 0,&\mbox{otherwise.}\end{array}\right.
\end{equation}
Note that if there are two isomorphic models $M_1,M_2\in G(\d)$ containing $\a<\o_2$ then Lemma \ref{t9} implies that $\pi_{M_1}(\a)=\pi_{M_2}(\a)$, so each $f_{\a}$ is well defined. Also, if $\a\neq \b$, then there are some $\d<\o_1$ and $M\in G(\d)$ such that $\set{\a,\b}\in M$, but then $\pi_M(\a)\neq\pi_M(\b)$ (i.e. $f_{\a}(\d)\neq f_{\b}(\d)$). So for $\a\neq\b$ we have $f_{\a}\neq f_{\b}$. If we denote the set of functions coding branches in $T$ by $\mathcal F=\set{f_{\a}:\a<\o_2}$ then $\mathcal F_{\a}=\set{f_{\a}\rest\d:\d<\o_1}$ will be the $\a$-th branch and the Kurepa tree will be given by $T=\bigcup_{\d<\o_1}T_{\d}$, where $T_{\d}=\mathcal F\rest \d=\set{f_{\a}\rest\d:\a<\o_2}$. We will show that for each $\d$, the level $T_{\delta}$ is countable. This will finish the proof. So assume that there are some $p\in \cg$ and $\d'<\o_1$ such that $p\f "\dot{T_{\d'}}\ \mbox{is uncountable}"$. Take a countable elementary submodel $M$ of $H_{({2^{\theta}})^+}$ such that $p,\cp,\d'\in M$ and denote $\d=M\cap \o_1$. Because we have chosen $M$ so that $\d'\in M$ we have $\d'<\d$. Consider the $(M,\cp)$-generic condition $p_{M}=p\cup\set{\seq{\d,\set{M\cap H_{\theta}}}}\le p$ (note that $p_M\le p$ because $p\in M$). The following claim shows that the $(M,\cp)$-generic condition forces that the branches passing through the $\d'$-th level are indexed only by ordinals less then $\o_2$ which are already in $M$.

\begin{claim}\label{t5}
Suppose that $p'\in \cp$ is such that $M'\in G(\d)\cap p'(\d)$ and $\d\ge\d'$. Then $p'\f \dot{T_{\d'}}=\set{\dot{f_{\a}}\rest\check\d':\check\a\in \check M'\cap \check\o_2}$
\end{claim}

\begin{proof}
The inclusion "$\supseteq$" is trivial. To prove the reverse inclusion take some $q\le p'$ and $\a'<\o_2$. Because $q\le p'$ we have that $q\f \dot{f_{\a'}}\rest\check\d'\in\dot{T_{\d'}}$. In order to finish the proof of the claim, we will find $r\le q$ and $\a\in M'\cap \o_2$ such that $r\f \dot{f_{\a'}}\rest\check\d'=\dot{f_{\a}}\rest\check\d'$. We will consider two cases.

\vskip2mm
Case I, $f_{\a'}\rest\d'=0$. Then for $0\in M'\cap \o_2$ we have that $f_0\rest\d'=f_{\a'}\rest\d'$. To see this notice that either for every $\g<\d'$ there is some $N\in G(\g)$ and then clearly $0\in N$ and $\pi_N(0)=0$ or $G(\g)=\ps$ and again $f_0(\g)=0$. So we clearly have $q\f \dot{f_{\a'}}\rest\check\d'=\dot{f_0}\rest\check\d'$.

\vskip2mm
Case II, there is $\g<\d'$ such that $f_{\a'}(\g)\neq 0$. Let $\g_0$ be minimal such $\g$. This means that for some $N_0\in G(\g_0)$ we have $\a'\in N_0$. Now, there is some $r\le q\le p'$ such that $N_0\in r(\g_0)$, hence there is some $N_1\in r(\d)$ such that $N_0\in N_1$ (note that $r(\d)$ is not empty because $M'\in p'(\d)\sse r(\d)$). Now we have that $\a'\in N_1$ ($N_0\in N_1$ implies $N_0\sse N_1$). Because $M'\in r(\d)$ there is an isomorphism $\varphi:N_1\izo M'$. Denote $\a=\varphi(\a')$. We show that the condition $r_1=\cl_{\d}(r)$ (note $r_1\le r\le q$) forces $"\dot{f_{\a'}}\rest\check\d'=\dot{f_{\a}}\rest\check\d'"$. First, if $\g'<\g_0$ then $r_1\f \dot{f_{\a}}(\check\g')=0=\dot{f_{\a'}}(\check\g')$. This is true because if there is some $N'\in r_1(\g')$ such that $\a=\varphi(\a')\in N'$, then there would be some $\varphi^{-1}(N')\in r_1(\g')$ such that $\a'\in \varphi^{-1}(N')$ which is impossible by the choice of $\g_0$ (it is a minimal ordinal such that $\a'$ belongs to some model on its level in generic filter). If $\g\ge\g_0$ and $\g<\d'$, let $r_1\f \dot{f_{\a'}}(\check\g)=\check\xi\neq 0$ (it cannot be 0 because $\a'\in N_0\in r_1(\g_0)$ so for every $\g\ge \g_0\ \exists N\in r_1(\g)\ \a'\in N_0\in N$ and since $\a'\neq 0$ we clearly have $\pi_N(\a')\neq 0$) and take $N'\in r_1(\g)$ such that $\a'\in N'$ and $\pi_{N'}(\a')=\xi$. Then $\a\in \varphi(N')\in r_1(\g)$ and clearly $\pi_{\varphi(N')}(\a)=\xi$ which implies $r_1\f \dot{f_{\a}}(\check\g)=\check\xi$. So $r_1\le q$ forces $"\dot{f_{\a'}}\rest\check\d'=\dot{f_{\a}}\rest\check\d'"$ and the claim is proved.
\end{proof}

Now according to the Claim \ref{t5} $p_{M}\f \dot{T_{\d'}}=\set{\dot{f_{\a}}\rest\check\d':\check\a\in \check M\cap \check\o_2}$ so $p$ cannot force that $T_{\d'}$ is uncountable (because $M$ is countable and $p_M\le p$), hence $T_{\d'}$ is countable in $V[\cg]$.
\end{proof}

\begin{corollary}\label{t19}
Every uncountable downward closed set $S\sse T$ contains a branch of $T$.
\end{corollary}

\begin{proof}
First recall that we denoted a branch of $T$ by $\mathcal F_{\a}=\set{f_{\a}\rest\g:\g<\o_1}$. Take a $\cp$-name $\dot S$ for $S$ and $p\in \mathcal G$ such that $p\f "\dot S$ is downward closed". Now pick $M\prec H_{(2^{\theta})^+}$ such that $\dot S,\cp,p\in M$ and denote $\d=M\cap \o_1$. So we have that $S\in M[\cg]\prec H_{(2^{\theta})^+}[\cg]=H_{(2^{\theta})^+}^{V[\cg]}$ (according to Lemma \ref{t18}). We have already shown that the condition $p_M=p\cup\set{\seq{M\cap \o_1,\set{M\cap H_{\theta}}}}$ is $(M,\cp)$-generic and by Claim \ref{t5} we have $p_M\f \dot{T_{\d}}=\set{\dot{f_{\a}}\rest\check\d:\check\a\in \check M\cap \check\o_2}$. Note also that according to Lemma \ref{t13} $p_M\f \check\d=\check M[\dot{\cg}]\cap \check\o_1$. Now, because $S$ is uncountable and each level in $T$ is countable by Theorem \ref{t4}, there is some $\b>\d$ such that $S\cap T_{\b}\neq\ps$ and because it is downward closed there is some $\a\in M\cap \o_2$ such that $f_{\a}\rest\d\in S$. Now the fact that $S$ and $\mathcal F_{\a}$ are in $M[\cg]$ (for $S$ is clear and for $\mathcal F_{\a}$ it follows from Lemma \ref{t25} and the fact that $f_{\a}$ is defined only from $\cg$ and $\a\in M$) implies that $S\cap \mathcal F_{\a}\in M[\cg]$. If the set $S\cap \mathcal F_{\a}$ is countable, then by elementarity of $M[\cg]$ in $H_{(2^{\theta})^+}^{V[\cg]}$ and Lemma \ref{t23} we have $S\cap \mathcal F_{\a}\sse M[\cg]$. Now $\d=M[\cg]\cap\o_1=M\cap\o_1$ and Lemma \ref{t24} imply that there is some $\g'<\d$ such that for every $\g\ge\g'$ we have $f_{\a}\rest\g\notin S$ which contradicts the assumption that $f_{\a}\rest\d\in S$. So $S\cap \mathcal F_{\a}$ is uncountable and if a downward closed set in a tree of height $\o_1$ intersects a branch at uncountably (hence cofinally) many levels it clearly contains that branch. Hence, $\mathcal F_{\a}\sse S$.
\end{proof}

\section{Almost Souslin tree}\label{continuousmatrix}

In this section we consider the slightly modified version of the poset $\cp$. Namely, let $\cp_c$ be the partial order satisfying all the conditions (1)-(3) from Definition \ref{d1} together with
\begin{itemize}
\item[(4)] for every $p\in \cp_c$ there is a continuous $\in$-chain $\seq{M_{\xi}:\xi<\o_1}$ (i.e. if $\b$ is a limit ordinal, then $M_{\b}=\bigcup_{\xi<\b}M_{\xi}$) of countable elementary submodels of $H_{\theta}$ such that $\forall\xi\in\supp(p)\ \ M_{\xi}\in p(\xi)$.
\end{itemize}
We point out that in this case the generic filter in $\cp_c$ is denoted by $\cg_c$ and that the function $G_c$ analogous to $G$ is a total function from $\o_1$ to $H_{\theta}$. Moreover, the poset $\cp_c$ is strongly proper. The proof of this fact needs slight modification of the proof of the Lemma \ref{T1}. Also, the Kurepa tree from the previous paragraph would be obtained in the same way with the poset $\cp_c$. Now we prove that the tree $T$ which we already constructed is an almost Souslin tree in $V[\cg_c]$, i.e. if $X\sse T$ is an antichain, then $L(X)=\set{\gamma<\o_1: X\cap T_{\g}\neq\ps}$ (the level set of $X$) is not stationary in $\o_1$. Hence, to show that $T$ is almost Souslin we have to find a club $\Gamma$ in $\o_1$ such that $\Gamma\cap L(X)=\ps$. So let $\tau\in H_{\theta}$ be a $\cp_c$-name and define the set
\[\textstyle
\Gamma_{\tau}=\set{\g<\o_1: \exists M\in G_c(\g)\ \tau\in M\ \&\ M[\cg_c]\cap\o_1=M\cap\o_1=\g}.
\]

\begin{lemma}[CH]\label{t27}
The set $\Gamma_{\tau}$ is a club in $\o_1$.
\end{lemma}

\begin{proof}
First we prove that $\Gamma_{\tau}$ is unbounded in $\o_1$. Take $\g_1<\o_1$ and assume that there is some $p\in \cg_c$ such that $p\f\forall \check\g\in \dot{\Gamma_{\tau}}\ \check\g<\check\g_1$. Take elementary submodel $M\prec H_{(2^{\theta})^+}$ such that $p,\cp_c,\g_1,\tau\in M$. The condition $p_M\le p$ given by $p_M=p\cup\set{\seq{M\cap\o_1,\set{M\cap H_{\theta}}}}$ is in $\cp_c$. To see this note that because $p\in \cp_c$ there is a continuous chain $\seq{M_{\xi}:\xi<\o_1}$ witnessing that condition (4) is satisfied, so by elementarity of $M$ there is a continuous chain $\seq{N_{\xi}:\xi<\d_M}$ such that $\forall \a\in \supp(p)\ N_{\a}\in p(\a)$ and $\bigcup_{\xi<\d_M}N_{\xi}=M$. The chain $\seq{N'_{\xi}:\xi<\o_1}$ witnessing that $p_M\in \cp_c$ is now recursively given by $\forall \xi<\d\ N'_{\xi}=N_{\xi}$, $N'_{\d}=M\cap H_{\theta}$, for succesor $\a+1>\d$ define $N'_{\a+1}\prec H_{\theta}$ as arbitrary submodel containing $N_{\a}$ while for limit $\a>\d$ define $N'_{\a}=\bigcup_{\xi<\a}N'_{\xi}$. Also, $p_M$ is an $(M,\cp_c)$-generic condition, so we have that $p_M\f \check M[\cg_c]\cap\check\o_1=\check M\cap\check\o_1$ (see \cite[Lemma III 2.6]{proper}). This implies that $p_M\f \check M\cap \check\o_1\in \dot{\Gamma_{\tau}}$, but from $\g_1\in M$ it follows that $p_M\f \check\g_1<\check M\cap \check\o_1$, which is in contradiction with the choice of $p$. So $\Gamma_{\tau}$ is unbounded in $\o_1$.

In order to prove that $\Gamma_{\tau}$ is a club, it is enough to show that for every ordinal $\d$ such that $\d=\sup(\Gamma_{\tau}\cap\d)$ there is some $M\in G_c(\d)$ which satisfies $M[\cg_c]\cap \o_1=M\cap\o_1=\d$. Because $\d=\sup(\Gamma_{\tau}\cap \d)$ there is an increasing sequence $\g_n\in \Gamma_{\tau}\cap \d$ such that $\sup_{n<\o}\g_n=\d$. Pick $N_{\g_0}\in G_c(\g_0)$ such that $\tau\in N_{\g_0}$, hence there is some $M\in G_c(\d)$ such that $\tau\in M$. First we will show that $M\cap \o_1=\d$. Inductively pick models $N_{\g_n}\in G_c(\g_n)$ such that $N_{\g_{n-1}}\in N_{\g_{n}}$. Note that because $\g_n\in \Gamma_{\tau}$, we have $N_{\g_n}\cap \o_1=\g_n$. Now, because $\sup_{n<\o}\g_n=\d$ and $\forall n<\o\ \g_n<M\cap\o_1$ (from $\forall n<\o\ \exists N\in G_c(\d)\ N_{\g_n}\in N$) we have that $M\cap \o_1\ge \d$. So assume that $M\cap \o_1=\b>\d$. Let $p\in \cg_c$ be any condition such that $M\in p(\d)$.

\begin{claim}\label{t12}
The set $D_{\d}=\set{q\le p: \exists \g'<\d\ \exists N\in q(\g')\ N\cap \o_1>\d}$ is dense below $p$ in $\cp_c$.
\end{claim}

\begin{proof}
Take arbitrary $p'\le p$ and pick a continuous $\in$-sequence $\seq{M_{\xi}:\xi<\o_1}$ such that $\forall \xi\in \supp(p')\ M_{\xi}\in p'(\xi)$. Because this chain is continuous, $\d$ is a limit ordinal and $M\in p'(\d)$ is such that $M\cap \o_1=\b>\d$, there is some $\xi_0<\d$ such that $M_{\xi_0}\cap \o_1>\d$. Now pick any $\xi_1\in \d\setminus \max(\d\cap (\supp(p')\cup \set{\xi_0}))$. Clearly $M_{\xi_1}\cap \o_1>\d$. To extend $p'$ to some $q\in D_{\d}$ first denote the ordinal $\a=\max(\supp(p')\cap \d)$. Further, note that for each $N\in p'(\a)$ there is some $M_N\in p'(\d)$ such that $N\in M_N$. Let $\psi_{N}:M_N\izo M_{\d}$ be isomorphism for each $N\in p'(\a)$. Because $\seq{M_{\xi}:\xi<\o_1}$ is a continuous chain and $\d$ is a limit ordinal, there is some $\xi_2<\d$ such that $M_{\xi_1}\in M_{\xi_2}$ and that moreover $\forall N\in p'(\a)\ \psi_{N}(N)\in M_{\xi_2}$. Now define
\[\textstyle
q=p'\cup\set{\seq{\xi_2,\set{M_{\xi_2}}\cup\set{\psi^{-1}_N(M_{\xi_2}):N\in p'(\a)\cap M_{\d}}}}.
\]
It is clear that $q\in \cp$ and the sequence $\seq{M_{\xi}:\xi<\o_1}$ witnesses that $q$ satisfies the property (4), hence $q$ is also in $\cp_c$.
\end{proof}
Now, according to Claim \ref{t12}, there is some $q\in \cg_c\cap D_{\d}$ below $p$. But this implies that there is some $N\in G_c(\g')$ for $\g'<\d$, satisfying $N\cap \o_1>\d$ which is impossible by the choice of $\g_n$ (note that $\g_n$ is cofinal in $\d$ and we would have that there is some $\g_{n'}>\g'$ such that $N_{\g_{n'}}\cap\o_1<N\cap\o_1$ and $N_{\g_{n'}}\in G_c(\g_{n'})$ and $N\in G_c(\g')$). So $M\cap\o_1=\d$.

We still have to prove that also $M[\cg_c]\cap\o_1=\d$. Let $\sigma\in M$ be a $\cp_c$-name for a countable ordinal. Because CH holds, the poset $\cp_c$ is $\o_2$-c.c., Lemma \ref{t15}, so we can assume that $\sigma$ is a nice name of cardinality at most $\o_1$, i.e. $\sigma=\set{\seq{\check\xi,p_{\xi}}:\xi<\o_1}$ where $\set{p_{\xi}:\xi<\o_1}$ is a maximal antichain in $\cp_c$. Now let $p\in \cg_c$ be any condition containing $M$. Because $p\in\cp_c$, there is a continuous chain $\seq{M_{\xi}:\xi<\o_1}$ such that $\forall \xi\in\supp(p)\ M_{\xi}\in p(\xi)$. Now there is an isomorphism $\varphi:M\izo M_{\d}$ and in the same way as in the proof of Claim \ref{t12} we show that there is some $q\in \cg_c$ and $\xi_1<\d$ such that $\varphi(\sigma)\in M_{\xi_1}\in q(\xi_1)\sse G_c(\xi_1)$. Because $\d=\sup(\Gamma_{\tau}\cap \d)$ we can assume that $\xi_1\in \Gamma_{\tau}$. Now according to Lemma \ref{T2} and the form of $\sigma$ and $\varphi(\sigma)$ we have that $\int_{\cg_c}(\sigma)=\int_{\cg_c}(\varphi(\sigma))<M_{\xi_1}[\cg_c]\cap \o_1=M_{\xi_1}\cap \o_1<\d$. Hence $M[\cg_c]\cap \o_1=\d$ and the proof is finished.
\end{proof}

\begin{theorem}[CH]\label{t28}
The tree $T$ is an almost Souslin tree.
\end{theorem}

\begin{proof}
Let $\tau'$ be a $\cp_c$-name for an antichain $X$ in $T$. Because CH holds in $V$, according to Lemma \ref{t15} $\cp_c$ is $\o_2$-c.c. so there is a $\cp_c$-name $\tau$ for $X$ which is in $H_{\theta}$. To prove the theorem, we will show that $L(X)\cap\Gamma_{\tau}=\ps$.

So assume that $X\cap T_{\d}\neq\ps$ for some $\d\in\Gamma_{\tau}$. Because $\d\in \Gamma_{\tau}$ there is some $M\in G_c(\d)$ such that $\tau\in M$ and that $M[\cg_c]\cap\o_1=M\cap\o_1=\d$, so take any $p\in\cg_c$ such that $M\in p(\d)$. Now, in the same way as in the proof of Claim \ref{t5} we know that $p$ forces that $T_{\d}=\set{f_{\a}\rest\d:\a\in M\cap \o_2}$, hence there is some $\a\in M\cap\o_2$ such that $f_{\a}\rest\d\in X$. Consider the branch $\mathcal F_{\a}$. It is defined solely from $\a\in M$ and $\cg_c$, so $\mathcal F_{\a}\in M[\cg_c]$. Also, because $\tau\in M$ we have that $X\in M[\cg_c]$. Consequently $\mathcal F_{\a}\cap X\in M[\cg_c]$, and from the fact that $X$ is an antichain it follows that this intersection is singleton (i.e. $f_{\a}\rest\d$). But, because $M[\cg_c]\prec H_{\theta}^{V[\cg_c]}$ (see Lemma \ref{t18}) and the height of $f_{\a}\rest\d$ is less than $\o_1$, there must exist some element $t\in \mathcal F_{\a}\cap X\cap M[\cg_c]$ which is of height less then $\d=M[\cg_c]\cap\o_1$. But then $t<f_{\a}\rest\d$ and both $t,f_{\a}\rest\d\in X$ which is in contradiction with the fact that $X$ is an antichain.
\end{proof}

\section{Concluding remarks}

We have seen that both versions of the matrix posets force the Continuum Hypothesis. It turns out that a bit more is true, both versions of the matrix poset force the combinatorial principle $\Diamond$ independently of the status of the Continuum Hypothesis in the ground model. In fact, if CH fails in $V$ that $\Diamond$ is forced follows from a slight adaptation of a result of Roslanowski and Shelah \cite{roslanowski} to our matrix posets that do not have cardinality $2^{\aleph_0}$ but are neverltheles $2^{\aleph_0}$-centered in the canonical way. So we may concentrate on the case that the ground model $V$ satisfies CH. If CH holds in $V$, then the following theorem from \cite{note2} proves a bit more for the continuous version of the matrix forcing (poset $\mathcal{P}_c$ of Section \ref{continuousmatrix}).

\begin{theorem}[CH]\label{t40}
$\Diamond^+$ holds in $V[\cg_c]$.
\end{theorem}

For the convenience of the reader we include the sketch of the proof.

\begin{proof}
First denote the generic club by $\Delta$, i.e. $\Delta$ is the club contained in the set $\tr(\cg_c)=\set{M\cap \o_1:M\in \bigcup_{p\in \cg_c}\ran(p)}$. The key part of the proof is the following claim which shows that the generic club is almost contained in every club from the ground model.

\begin{claim}\label{t41}
Let $C\sse \o_1$ be a club in $V$. Then there is a countable ordinal $\d$ such that for every $\b\ge \d$ if $\b\in\Delta$ then $\b\in C$.
\end{claim}

\begin{proof}
Because $C\in V$ there is some $\a<\o_1$ such that $\exists M\in G_c(\a)\ C\in M$. Denote $\d=M\cap \o_1$. By elementarity of $M$ we know that $C$ is a club in $\d$, so $\d\in C$. Now take arbitrary $\b>\d$ such that $\b\in \Delta$. This means that for some $\g<\o_1$ ($\a<\g$) there is some $N'\in G_c(\g)$ such that $N'\cap \o_1=\b$. Also, by the definition of $\cp_c$ there is some $N\in G_c(\g)$ such that $C\in N$. Now by elementarity of $N$ we conclude that $C$ is a club in $\b=N\cap \o_1$ so $\b\in C$.
\end{proof}

\begin{claim}\label{t43}
There is, in $V[\cg_c]$, a sequence $\seq{S_{\a}:\a<\o_1}$ such that for every $\a<\o_1$ we have $S_{\a}\in [P(\a)]^{\le \o}$ and that
\[\textstyle
\forall X\in P(\o_1)\cap V\ \exists \g<\o_1\ \forall \a\ge \g\ X\cap \a\in S_{\a}.
\]\end{claim}

\begin{proof}
By CH in $V$ we can find an increasing continuous sequence of countable sets $\seq{D_{\a}:\a<\o_1}$ such that $D_{\a}\sse P(\a)$ and $\bigcup_{\a<\o_1}D_{\a}=[\o_1]^{\le\o}$. Let $f:\o_1\to \o_1$ be defined by $f(\a)=\min(\Delta\setminus (\a+1))$. Finally, for $\a<\o_1$ we let $S_{\a}=\set{X\cap \a:X\in D_{f(\a)}}$. To see that the sets $S_{\a}$ ($\a<\o_1$) satisfy the statement of the claim pick any $X\sse \o_1$ which is in $V$. As in the proof of Lemma 2.1 in \cite{devlin} there is a club $E\sse \o_1$ which is in $V$ such that $\forall \a\in E\ \forall \b<\a\ X\cap\b\in D_{\a}$. Now according to Claim \ref{t41} there is some $\g<\o_1$ such that $\Delta\setminus \g\sse E$. For $\g\le \a<\o_1$ it holds $f(\a)\in \Delta\setminus\a\sse\Delta\setminus\g\sse E$ so as $f(\a)>\a$, the choice of $E$ ensures that $X\cap \a\in D_{f(\a)}$. So the claim is proved.
\end{proof}

Let $\seq{S_{\a}:\a<\o_1}$ be the sequence from the previous claim. For $\a<\o_1$ let $W_{\a}=P(\a)\cap (\bigcup_{X\in S_{\a}}L_{\a+2}[X,\Delta\cap \a])$. Then $\seq{W_{\a}:\a<\o_1}$ is a $\Diamond^+$ sequence. To show this pick arbitrary $A\sse \o_1$. Because $\cp_c$ is $\o_2$-c.c. there is a name for $A$ which is coded by some $X\sse \o_1$. Hence, $A\in L[X,\Delta]$. By Claim \ref{t43} there is a $\g<\o_1$ such that $\forall \a\ge \g\ X\cap \a\in S_{\a}$. By induction we define a normal sequence $\seq{\a_{\xi}:\xi<\o_1}$ in $\o_1$. Let $\a_0>\g$ be the least $\a$ such that $L_{\a}[X\cap \a,\Delta\cap \a]\prec L_{\o_1}[X,\Delta]$. If $\a_{\xi}$ is defined let $\a_{\xi+1}>\a_{\xi}$ be the least ordinal such that $L_{\a_{\xi+1}}[X\cap \a_{\xi+1},\Delta\cap \a_{\xi+1}]\prec L_{\o_1}[X,\Delta]$. Let $B=\seq{\a_{\xi}:\xi<\o_1}$. Then it is easily checked by the construction that $B$ is a club in $\o_1$. So pick arbitrary $\a\in B$ (we will prove that $A\cap \a, B\cap \a\in W_{\a}$). Because $\a>\g$ we have $X\cap \a\in S_{\a}$, so $P(\a)\cap L_{\a+2}[X\cap \a,\Delta\cap \a]\sse W_{\a}$. Because $L_{\a}[X\cap \a,\Delta\cap \a]\prec L_{\o_1}[X,\Delta]$ we have that $A\cap \a$ is first-order definable over $L_{\a}[X\cap \a,\Delta\cap \a]$ so $A\cap \a\in L_{\a+1}[X\cap \a,\Delta\cap \a]\sse W_{\a}$. Similarly we would show that $B\cap \a\in W_{\a}$ and the theorem is proved.
\end{proof}

It is clear that this proof adapts to showing that the original matrix poset (poset $\mathcal{P}$ of Section \ref{matrix}) also forces $\Diamond.$


\footnotesize
\begin{thebibliography}{22}

\bibitem{aspero}
	D.\ Aspero, M.\ A.\ Mota,
	Forcing consequences of PFA together with the continuum large,
    Trans. Amer. Math. Soc., published electronically,
	http://dx.doi.org/10.1090/S0002-9947-2015-06205-9

\bibitem{aspero1}
	D.\ Aspero,
	A forcing notion collapsing $\aleph_3$ and preserving all other cardinals,
	preprint
	(2014).

\bibitem{golshani}
	M.\ Golshani,
	Almost Souslin Kurepa trees,
	Proc. Amer. Math. Soc.
	141,5 (2013) 1821--1826.

\bibitem{devlin}
	K.\ Devlin,
	Concerning the consistency of CH+SH,
	Ann. Math. Logic.
	19 (1980) 115,125.

\bibitem{mitchell}
	W.\ Mitchell,
	$I[\o_2]$ can be the nonstationary idel on $\cf(\o_1)$,
	Trans. Amer. Math. Soc.
	361,2 (2009) 561-601.

\bibitem{neeman}
	I.\ Neeman,
	Forcing with sequences of models of two types,
	Notre Dame J. Formal Logic
	55 (2014) 265--298.

\bibitem{neemanslides}
	I.\ Neeman,
	Higher analog of the proper forcing axiom,
	Talk at the Fields Institute, Toronto, October 2012.

\bibitem{roslanowski}
   A.\ Roslanowski, S.\ Shelah,
   More forcing notions imply diamond,
   Archive Math. Logic
   35 (1996) 299-313.

\bibitem{proper}
	S.\ Shelah,
	Proper and improper forcing,
	Springer, 1998.

\bibitem{pfa}
	S.\ Todorcevic,
	A note on the proper forcing axiom,
	In Axiomatic set theory, (Boulder, Colorado 1983),
	volume 31 of Contemporary Mathematics, pages 209--218,
	Amer. Math. Soc. 1984,
	ed. J.\ Baumgartner, D.\ Martin and S.\ Shelah.

\bibitem{note}
	S.\ Todorcevic,
	Kurepa tree with no stationary antichain,
	Note of September, 1987.

\bibitem{note2}
	S.\ Todorcevic,
	Forcing club with finite conditions,
	Note of December, 1982.

\end{thebibliography}
\end{document}